\documentclass[12pt,a4paper]{amsart}
\usepackage{amsfonts, amsthm, amssymb,verbatim,amscd,amsmath,blindtext}
\usepackage{multirow}
\usepackage{graphicx}
\usepackage{enumitem}
\usepackage{amssymb}
\usepackage{ytableau}
\usepackage{tikz}
\usepackage{tikz-cd}
\usepackage{float, pifont}
\usepackage{mathtools}

\DeclareMathOperator{\link}{link}

\def\NZQ{\mathbb}               

\def\ZZ{{\NZQ Z}}

\def\Ac{{\mathcal A}}

\def\opn#1#2{\def#1{\operatorname{#2}}} 
\opn\chara{char} \opn\length{\ell} \opn\pd{pd} \opn\rk{rk}
\opn\projdim{proj-dim}
\opn\injdim{inj\,dim} \opn\rank{rank}
\opn\depth{depth} \opn\grade{grade} \opn\height{height}
\opn\embdim{emb\,dim} \opn\codim{codim}
\opn\Cl{Cl}

\opn\Tr{Tr} \opn\bigrank{big\,rank}
\opn\superheight{superheight}\opn\lcm{lcm}
\opn\trdeg{tr\,deg}
\opn\rdeg{rdeg}
	\opn\reg{reg} \opn\lreg{lreg} \opn\ini{in} \opn\lpd{lpd}
	\opn\size{size} \opn\sdepth{sdepth}
	\opn\link{link}\opn\fdepth{fdepth}\opn\lex{lex}
	\opn\tr{tr}
	\opn\type{type}
	\opn\gap{gap}
	\opn\arithdeg{arith-deg}
	\opn\revlex{revlex}
	\opn\div{div} \opn\Div{Div} \opn\cl{cl} \opn\Cl{Cl}
	\opn\Spec{Spec} \opn\Supp{Supp} \opn\supp{supp} \opn\Sing{Sing}
	\opn\Ass{Ass} \opn\Min{Min}\opn\Mon{Mon}
	\opn\Ann{Ann} \opn\Rad{Rad} \opn\Soc{Soc}
	\opn\Im{Im} \opn\Ker{Ker} \opn\Coker{Coker} \opn\Am{Am}
	\opn\Hom{Hom} \opn\Tor{Tor} \opn\Ext{Ext} \opn\End{End}
	\opn\Aut{Aut} \opn\id{id}
	
	\opn\nat{nat}
	\opn\pff{pf}
	\opn\Pf{Pf} \opn\GL{GL} \opn\SL{SL} \opn\mod{mod} \opn\ord{ord}
	\opn\Gin{Gin} \opn\Hilb{Hilb}\opn\sort{sort}
	\opn\PF{PF}\opn\Ap{Ap}
	\opn\mult{mult}
	\opn\bight{bight}
	\opn\div{div}
	\opn\Div{Div}
	\opn\aff{aff}
	\opn\relint{relint} \opn\st{st}
	\opn\lk{lk} \opn\cn{cn} \opn\core{core} \opn\vol{vol}  \opn\inp{inp} \opn\nilpot{nilpot}
	\opn\link{link} \opn\star{star}\opn\lex{lex}\opn\set{set}
	\opn\width{wd}
	\opn\Fr{F}
	\opn\QF{QF}
	\opn\G{G}
	\opn\type{type}\opn\res{res}
	\opn\conv{conv}
	\opn\Int{Int}
	\opn\Deg{Deg}
	\opn\Sym{Sym}
	\opn\Con{Con}
	\opn\gr{gr}
	
	\def\pot#1#2{#1[\kern-0.28ex[#2]\kern-0.28ex]}

	\opn\dirlim{\underrightarrow{\lim}}
	\opn\inivlim{\underleftarrow{\lim}}
	\def\Implies{\ifmmode\Longrightarrow \else
		\unskip${}\Longrightarrow{}$\ignorespaces\fi}
	\def\implies{\ifmmode\Rightarrow \else
		\unskip${}\Rightarrow{}$\ignorespaces\fi}
	\def\iff{\ifmmode\Longleftrightarrow \else
		\unskip${}\Longleftrightarrow{}$\ignorespaces\fi}

	\let\:=\colon
	\newtheorem{Theorem}{Theorem}[section]
	\newtheorem{TheoremA}{Theorem}
	
	\newtheorem{Lemma}[Theorem]{Lemma}
	
	\newtheorem{Proposition}[Theorem]{Proposition}
	
	\theoremstyle{definition}
	\newtheorem{Remark}[Theorem]{Remark}

	\let\epsilon\varepsilon
	\let\kappa=\varkappa
	\opn\dis{dis}
	\def\pnt{{\raise0.5mm\hbox{\large\bf.}}}
	
	\opn\Lex{Lex}

\begin{document}
\title[A special class of pure $O$-sequences]{A special class of pure $O$-sequences}
\author[T\`ai Huy H\`a]{T\`ai Huy H\`a}
\address{Mathematics Department, Tulane University, 6823 St. Charles Avenue, New Orleans, LA 70118, USA}
\email{tha@tulane.edu}
\author[Takayuki Hibi]{Takayuki Hibi}
\address{Department of Pure and Applied Mathematics, Graduate School of Information Science and Technology, Osaka University, Suita, Osaka 565-0871, Japan}
\email{hibi@math.sci.osaka-u.ac.jp}
\author[Fabrizio Zanello]{Fabrizio Zanello}
\address{Department of Mathematical Sciences, Michigan Tech, Houghton, MI 49931, USA}
\email{zanello@mtu.edu}
\dedicatory{}
\begin{abstract}
The pure $O$-sequences of the form $(1,a,a,\ldots)$ are classified.
\end{abstract}	
\maketitle

\thispagestyle{empty}
\section*{Introduction}
Let $x_1, \ldots, x_s$ represent distinct indeterminates with $\deg x_i = 1$, for $i = 1, \ldots, s$. A nonempty, finite set $\Ac$ of monomials in $x_1, \ldots, x_s$ is called an {\em order ideal of monomials} if for any $u \in \Ac$ and any monomial $v$ that divides $u$, we have $v \in \Ac$.  In particular, $1 \in \Ac$ for any order ideal of monomials $\Ac$.  We say that $\Ac$ is {\em pure} if the maximal elements of $\Ac$, with respect to divisibility, all have the same degree.
The {\em $h$-vector} of $\Ac$ is defined as $h(\Ac) = (h_0,h_1,\ldots,h_n),$ where
$$n = \max \{\deg u : u \in \Ac\} \text{{\ }and{\ }} h_i = \big|\{u \in \Ac : \deg u = i \}\big|, \text{ for } 0 \leq i \leq n.$$
Clearly, $h_0 = 1$.

A finite sequence of positive integers $h = (h_0,h_1,\ldots,h_n)$ is called an {\em $O$-sequence} if there exists an order ideal of monomials $\Ac$ with $h = h(\Ac)$.  Finally, following Stanley \cite{pure}, an $O$-sequence $h$ is {\em pure} if there exists a pure order ideal of monomials $\Ac$ with $h = h(\Ac)$. Equivalently, in the language of commutative algebra, pure $O$-sequences coincide with the Hilbert functions of (standard graded) artinian monomial level algebras. We refer to \cite{BMMNZ,MNZ} for an introduction to the theory of pure $O$-sequences both combinatorially and algebraically, and for some recent developments.

A classification of the possible $O$-sequences is essentially due to Macaulay (see \cite{Ma} and \cite[Theorem 2.2]{Stanley}).  On the other hand, an explicit characterization of pure $O$-sequences seems entirely out of reach, despite much effort by many researchers. A notably long-standing problem in this field is a conjecture of Stanley's, stating that the $h$-vector of any matroid complex is a pure $O$-sequence \cite{pure, S1996}. Partial results in this direction have been obtained in, for instance, \cite{cran, dall, DKK, HSZ, HLO, Hibi, klee, kook, merino, sch}.

The purpose of the present paper is to prove the following:

\begin{TheoremA}
\label{New_Orleans}
Let $n \geq 4$. Then a sequence $h=(1,a,a,\ldots,a,b) \in \ZZ_{>0}^{n+1}$ is a pure $O$-sequence if and only if $b \leq a \leq 2b$.
\end{TheoremA}

The proof of Theorem \ref{New_Orleans} is given in Section $1$.  In addition, in Section $2$, as a supplement to Theorem \ref{New_Orleans}, the pure $O$-sequences $(1,a,b)$ and $(1,a,a,b)$ are classified (see Proposition \ref{1ab1aab}).

Part of our motivation for considering pure $O$-sequences of the form $(1,a,a,\ldots,a,b)$ arises from the theory of $\delta$-vectors of Castelnuovo polytopes \cite{LBT, kawaguchi, SHEH}, where it was shown that a sequence $(1,a,a,\ldots,a,b) \in \ZZ_{>0}^{n+1}$, with $n \geq 2$, is the $h$-vector of a Cohen--Macaulay graded domain if $b \leq a \leq (b+1)(n+1)$.  In particular, Theorem \ref{New_Orleans} together with Proposition \ref{1ab1aab} (ii) guarantees that, when $n \geq 3$, any pure $O$-sequence $(1,a,a,\ldots,a,b)$ is the $h$-vector of a Cohen--Macaulay graded domain.

Finally, a problem of current interest in commutative algebra is to determine classes of artinian algebras that enjoy the so-called {\em Weak} (or {\em Strong}) {\em Lefschetz Properties} \cite{BMMNZ,har}. The results of this paper imply that, over a field of characteristic zero, {\em any} artinian monomial level algebra with Hilbert function given by a pure $O$-sequence of the form $(1,a,a,\ldots,a,b)$ has the Strong Lefschetz Property.

\section{Proof of Theorem \ref{New_Orleans}}
Our proof of Theorem \ref{New_Orleans} is divided into several lemmata.

\begin{Lemma}\label{ba}
\label{aaaaa}
Let $h=(1,a,a,\ldots,a,b) \in \ZZ_{>0}^{n+1}$ be a pure $O$-sequence, where $n \geq 3$. Then $b\leq a$.
\end{Lemma}

\begin{proof}
This result can easily be shown using \cite[Proposition 3.6]{BMMNZ}, but we present a self-contained proof since it appears of independent interest. Consider an artinian monomial level $k$-algebra $A=\bigoplus_{i=0}^nA_i$ with Hilbert function $h$, where $k=A_0$ is a field of characteristic zero. By Hibi-Hausel's $g$-theorem on the differentiability of a pure $O$-sequence through its first half (see \cite[Theorem 1.1]{Hibi} and \cite[Theorem 6.2]{Hausel}), we deduce that multiplication by a Zariski-general linear form $L$ between consecutive graded pieces $A_i$ and $A_{i+1}$ is injective, for all indices $i\leq \lfloor n/2 \rfloor$.

Now note that, because $A_i$ and $A_{i+1}$ have the same $k$-vector space dimension in those degrees (namely, $a$), multiplication by $L$ is in fact bijective. Finally, since the grading of $A$ is standard, it is easy to see (\cite[Proposition 2.1]{MMN}) that if multiplication by $L$ is surjective from some degree $i$ to $i+1$, then it is surjective from degree $j$ to $j+1$, for all  $j\geq i$. The case $j=n-1$ yields $b\leq a$.
\end{proof}

\begin{Lemma}
\label{bbbbb}
Let $n\geq 2$. Then $(1,a,a,\ldots,a,b) \in \ZZ_{>0}^{n+1}$ is a pure $O$-sequence for any $b=\lceil a/2\rceil,\ldots, a$.
\end{Lemma}

\begin{proof}
Partition $a$ into exactly $b$ parts of size $\leq 2$, say $a=2s_1+s_2$ with $s_1+s_2=b$, and consider the two sets of monomials $x_1x_2^{n-1}, \ldots,x_{2s_1-1}x_{2s_1}^{n-1}$ and $y_1^n,\ldots, y_{s_2}^n$. Then the pure $O$-sequence they generate is
\begin{eqnarray*}
&&(1,0,0,\ldots,0)+s_1\cdot (0,2,2,\ldots,2,1)+s_2\cdot (0,1,1,\ldots,1,1)\\
&=&(1,2s_1+s_2,2s_1+s_2,\ldots,2s_1+s_2,s_1+s_2)\\
&=&(1,a,a,\ldots,a,b),
\end{eqnarray*}
as desired.
\end{proof}

\begin{Lemma}\label{ccccc}
Let $n\geq 4$. Then $(1,a,a,\ldots,a,b) \in \ZZ_{>0}^{n+1}$ cannot be a pure $O$-sequence if $2b<a$.
\end{Lemma}

\begin{proof}
Let $x_1, \ldots, x_a$ be variables.  Suppose that $h=(1,a,a,h_3,\ldots,h_{n-1},b) \in \ZZ_{>0}^{n+1}$ is a pure $O$-sequence, where $2b<a$ and $n \geq4$.  Let $\{u_1, \ldots, u_b\}$ denote a set of monomials in $x_1, \ldots, x_a$ of degree $n$ that generates $h$.  Let $p_j$ be the number of variables $x_i$ for which $x_i$ divides $u_j$, but does not divide any of $u_1, \ldots, u_{j-1}$.  Since $2b < a$, we can assume that $p_1 \geq 3$.  Let $q_j$ be the number of quadratic monomials $x_ix_{i'}$ for which $x_ix_{i'}$ divides $u_j$, but does not divide any of $u_1, \ldots, u_{j-1}$.  Then
\begin{equation}\label{pq}
\sum_{j=1}^{b} p_j = \sum_{j=1}^{b} q_j = a.
\end{equation}

Note that $q_j \geq p_j$, for each $j$.  Furthermore, since $n \geq 4$, it follows that $q_1 > p_1$. This contradicts (\ref{pq}), completing the proof.
\end{proof}

Combining Lemmata \ref{aaaaa}, \ref{bbbbb}, and \ref{ccccc} proves Theorem \ref{New_Orleans}.

\begin{Remark} \label{rmk.h}
Let $n\geq 4$, and $h = (1,h_1,\ldots,h_n) \in \ZZ_{>0}^{n+1}$ be a pure $O$-sequence such that $h_1 = h_i$ for some $2 \leq i \leq n-2$. It follows from \cite[Theorem 1.1]{Hibi} that $h_1 = h_2$.

Let $h_1 = h_2 = a$ and $h_n = b$.  Arguing as in the proof of Lemma \ref{ccccc}, we have $a \leq 2b$ and each $q_j \in \{1,2\}$.  Let $q_j = 2$ for $1 \leq j \leq b'$, and $q_j = 1$ for $b'+1 \leq j \leq b$.  If $1 \leq j \leq b'$, then $u_j = x_{j_1} x_{j_2}^{n-1}$ with $j_1 \neq j_2$.  If $b'+1 \leq j \leq b$, then either $u_j = x_{j_1}^n$ or $u_j = x_{j_1} x_{j'_2}^{n-1}$, where $1 \leq j'_2 \leq b'$.  Thus, each $h_i = 2b' + (b - b') = a$.  Hence $h=(1,a,a,\ldots,a,b)$.
\end{Remark}

\begin{Remark}
The proof of Lemma \ref{ccccc}, together with Remark \ref{rmk.h}, shows that all pure $O$-sequences of the form $(1,a,a,\ldots,a,b)$ can be constructed starting from the two pure $O$-sequences $(1,2,2,\ldots,2,1)$ and $(1,1,\ldots,1,1)$.
\end{Remark}

\begin{Remark}
Interestingly from an algebraic standpoint,  it follows from the argument of Lemma \ref{ba} that, over a field of characteristic zero, any artinian monomial level algebra with Hilbert function given by the pure $O$-sequence $(1,a,a,\ldots,a,b)$ enjoys the Strong Lefschetz Property.
\end{Remark}

\section{The pure $O$-sequences $(1,a,b)$ and $(1,a,a,b)$}
As a supplement to Theorem \ref{New_Orleans}, we give the following result, which completes the characterization of pure $O$-sequences of the form $(1,a,\ldots,a,b)$.

\begin{Proposition}
\label{1ab1aab}
Let $a$ and $b$ be positive integers.
\begin{itemize}
\item[(i)] The sequence $(1,a,b)$ is a pure $O$-sequence if and only if $\lceil a/2\rceil \leq b\leq \binom{a+1}{2}$.
\item[(ii)] The sequence $(1,a,a,b)$ is a pure $O$-sequence if and only if $\lceil a/3\rceil \leq b\leq a$.
\end{itemize}
\end{Proposition}

\begin{proof}
(i) The proof is a simple exercise.  See \cite[Corollary 4.7]{BMMNZ} and \cite[Example 1.2]{Hibi}.

(ii) We want to determine when $b$ monomials of degree $3$ in $a$ variables have a total of $a$ degree 2 divisors. First note that each of these $b$ monomials can involve at most 3 variables. This implies $\lceil a/3\rceil \leq b$. The upper bound $b\leq a$ was proven in Lemma \ref{aaaaa}. Thus, it remains to show that $(1,a,a,b)$ is a pure $O$-sequence for each integer $b=\lceil a/3\rceil,\ldots, a$.

Any degree $3$ monomial is of one of the following three kinds: $xyz$ (which generates the pure $O$-sequence $(1,3,3,1)$); $xy^2$ (generating $(1,2,2,1)$); and $x^3$ (generating $(1,1,1,1)$). Further, any integer $a$ in the range under consideration can be partitioned into {\em exactly} $b$ parts of size $\leq 3$ (since, clearly, any integer $a-b$ satisfying $0\leq a-b\leq 2b$ can be partitioned into {\em at most} $b$ parts of size $\leq 2$). Hence write $a=3t_1+2t_2+t_3$, where the multiplicities $t_i$ are nonnegative and sum up to $b$.

Now consider $t_1$ squarefree monomials of degree 3 in disjoint sets of variables, say
$$x_1x_2x_3, {\ }\ldots,{\ } x_{3t_1-2}x_{3t_1-1}x_{3t_1};$$
$t_2$ monomials of the form
$$y_1y_2^2, {\ }\ldots,{\ } y_{2t_2-1}y_{2t_2}^2;$$
and  $t_3$ monomials of the form
$$z_1^3, {\ }\ldots,{\ } z_{t_3}^3.$$
The pure $O$-sequence generated by the above $t_1+t_2+t_3$ monomials is given by:
\begin{eqnarray*}
&&(1,0,0,0)+t_1\cdot (0,3,3,1)+t_2\cdot (0,2,2,1)+t_3\cdot (0,1,1,1)\\
&=&(1,3t_1+2t_2+t_3,3t_1+2t_2+t_3,t_1+t_2+t_3)\\
&=&(1,a,a,b).
\end{eqnarray*}
This concludes the proof of (ii).
\end{proof}

\begin{Remark}
We wrap up by noting that while any pure $O$-sequence is an artinian level Hilbert function \cite{BMMNZ}, the converse is far from being true, even for $n=2$.

It is easy to see that the sequence $(1,a,b)$ is level if and only if $1\leq b\leq \binom{a+1}{2}$. Also, $(1,a,a,b)$ is level for any $b$ in the range $1\leq b\leq a$.  More generally, using the techniques of \cite{Ia1}, it can be shown that for any $n\geq 3$, $(1,a,a,\ldots,a,b) \in \ZZ_{>0}^{n+1}$ is level whenever $1\leq b\leq a$, over a field of any characteristic.

However, larger values of $b$ may also be attained when $n\geq 3$. For instance, $(1,13,13,14)$ is a level Hilbert function (\cite[Chapter 3]{BMMNZ}). In fact, we remark here without proof that, in stark contrast to the case of pure $O$-sequences, it is possible to construct level sequences $(1,a,a,b)$ where the difference $b-a$ gets arbitrarily large and is asymptotic to $a$ itself. This strongly suggests that an explicit characterization of level Hilbert functions might hard to achieve, even in the special case $(1,a,a,b)$.
\end{Remark}

\section*{Acknowledgements} The main ideas contained in this paper were discussed during a visit to New Orleans by the second and third authors in spring 2024. They are sincerely grateful to the Tulane Math Department for its support and hospitality. The first and third authors were partially supported by a Simons Foundation grant (\#850912 and \#630401, respectively).

\end{document}